\documentclass[review]{elsarticle}
\usepackage{lineno,hyperref}
\modulolinenumbers[5]

\usepackage{amsmath,amscd}
\usepackage{amsthm}
\usepackage{amssymb}
\usepackage{graphics}
\usepackage{graphicx}
\usepackage{colortab} 
\usepackage{amscd}
\usepackage{pstricks}
\usepackage{pst-plot}
\usepackage{multido}
\usepackage{pst-coil}
\usepackage{colortab}
\usepackage[margin=1in]{geometry}

\usepackage{psfrag}

\DeclareMathOperator{\supp}{supp}
\theoremstyle{plain}
  \newtheorem{theorem}{\bf Theorem}[section]
  \newtheorem{proposition}[theorem]{\bf Proposition}
  \newtheorem{lemma}[theorem]{\bf Lemma}
  \newtheorem{corolario}[theorem]{\bf Corollary}

\theoremstyle{remark}
  \newtheorem{remark}[theorem]{\bf Remark}
  
\definecolor{darkgreen}{rgb}{0,.5,0}
\definecolor{verde}{rgb}{0.55,0.68,0.09}
\definecolor{pistacho}{rgb}{0.80,0.89,0.17}
\definecolor{verdeclaro}{rgb}{0.9,0.9,0.25}
\definecolor{marron}{rgb}{0.66,0.47,0.13}
\definecolor{rojo}{rgb}{0.68,0.13,0.04}
\definecolor{azul}{rgb}{0.38,0.53,0.72}
\definecolor{amarillo}{rgb}{0.85,0.81,0.00}
\definecolor{marronclaro}{rgb}{.749,.675,.376}
\definecolor{palegray}{rgb}{0.85,0.85,0.85}
\definecolor{gray}{rgb}{0.63,0.63,0.63}
\definecolor{darkgray}{rgb}{0.50,0.50,0.50}
  
\begin{document}
	\begin{frontmatter}
	\title{A new convergent algorithm to approximate potentials from fixed angle scattering data}

\author[mymainaddress]{Juan A. Barcel\'o}
\ead{juanantonio.barcelo@upm.es}
\author[mymainaddress]{Carlos Castro}
\ead{carlos.castro@upm.es}
\author[mysecondaryaddress]{Teresa Luque\corref{mycorrespondingauthor}}
\cortext[mycorrespondingauthor]{Corresponding author.}
\ead{t.luque@ucm.es}
\author[mymainaddress]{Mari Cruz Vilela}
\ead{maricruz.vilela@upm.es}

\address[mymainaddress]{Departamento de Matem\'atica e Inform\'atica aplicadas a las Ingenier\'ias Civil y Naval, Universidad Polit\'ecnica de Madrid}
\address[mysecondaryaddress]{Departamento de An\'alisis Matem\'atico y Matem\'atica aplicada, Facultad de Matem\'aticas, Universidad Complutense de Madrid. }
	\begin{abstract}
		We introduce a new iterative method to recover a real compact supported potential of the Schr\"o\-din\-ger operator from their fixed angle scattering data. The method combines a fixed point argument with a suitable approximation of the resolvent of the Schr\"o\-din\-ger operator by partial sums associated to its Born series. Convergence is established for potentials with small norm in certain Sobolev spaces. As an application we show some numerical experiments that illustrate this convergence. 
	\end{abstract}
	
	
\end{frontmatter}

\linenumbers


\section{Introduction and statement of results}
We consider the scattering problem for the Schr\"{o}dinger operator $-\Delta+q$ in $\mathbb{R}^d$, $d\ge 2$, where $q$ is a real valued potential with compact support in $B(0,R)$. Here $B(0,R)$ denotes the ball centred at the origin with radius $R>0.$ 

Associated to a given wave number $k>0$ and an incident direction $\theta\in S^{d-1}$ we consider the incident wave $u_i(x)=e^{ik\theta \cdot x}$.
Here $S^{d-1}$ denotes the unit sphere in $\mathbb{R}^d.$
The outgoing scattering solution $u=u(x, \theta, k)$ with wave number $k$ and direction of propagation $\theta,$ is the solution of equation
\begin{equation}\label{ecuacion}
\left(\Delta  + k^2\right) u(x)=q(x)u(x), \qquad x\in\mathbb{R}^d,
\end{equation}
which can be written as
$u=u_i+u_s$ with $u_s(x,\theta,k)$ satisfying the outgoing Sommerfeld radiation condition given by
\begin{equation}\label{Sommerfeld}
\partial_ru_s-iku_s= o\left(r^{-(d-1)/2}  \right ), \quad r=|x| \longrightarrow \infty.
\end{equation}

The function $u_s(x,\theta,k)$ named the scattered wave, is the perturbation of $u$ due to the potential. It is well known that for appropriate $q$, $u_s$ satisfies the following asymptotic expression as $|x|\rightarrow \infty$, 
\begin{equation}\label{asintotica}
u_s(x, \theta, k)=c_dk^{(d-3)/2}e^{ik |x|}|x|^{-(d-1)/2}u_\infty (\theta', \theta,k)+o\left( |x|^{-(d-1)/2}  \right),
\end{equation}
where $\theta'= x/|x|$ is the reflecting angle and 
\begin{equation}\label{amplitud}
u_\infty (\theta', \theta,k)=\int_{\mathbb{R}^d}e^{-ik \theta' \cdot y}q(y) u(y, \theta,k)dy.
\end{equation}
The function $u_{\infty}$ is called the \emph{scattering amplitude}  or \emph{far-field pattern}, and it represents the measurements in the inverse scattering problem. 
For a successful description of direct and inverse scattering problems we refer the reader to Chapter 5 in \cite{Es}.
We are interested in recovering the potential $q(x)$ from the knowledge of the scattering amplitudes  $u_\infty(\theta',\theta_0,k),$ for fixed incident direction $\theta_0$ and $(\theta', k) \in S^{d-1} \times (0, \infty)$, with $\theta'=x/|x|$. In fact, by fixed incident direction we mean data for both $\theta=\theta_0$ and $\theta=-\theta_0,$ since we are considering wave numbers $k>0.$ This problem is known as the fixed angle inverse scattering problem and appears naturally in quantum physics. In general, the recovery of information about $q$ from scattering amplitudes is known as inverse scattering problem and it has been studied by several authors. Here we mention some of them that we consider relevant.

Early works on the study of inverse scattering problems can be found in the middle of the last century with the results of Gelfand and Levitan \cite{GelfandLevitan}, Jost and Kohn \cite{JostKohn}, and Moses \cite{Moses}.
Later on, Prosser generalized the method in \cite{JostKohn} to recover the potential $q,$ based on the Born series (nonlinear approximations), when the Friederichs norm of $q$ is small enough (see \cite{P1,P2,P3,P4}). However, the procedures employed in these papers are purely formal, and the smallness condition is difficult to characterize. 

In the nineties, Eskin and Ralston (see \cite{EskinRalston,EskinRalston2,EskinRalston3}), and also Stefanov (see \cite{St}) studied the problem of uniqueness.
At this time, P\"aiv\"arinta, Somersalo and Serov introduced new techniques for dealing with the problem of singularities using all the scattering data
(see \cite{PaivarintaSomersalo,PaivarintaSerov,PaivarintaSerovSomersalo}).
This problem was also studied by Greenleaf and Uhlmann (see \cite{GreenleafUhlmann}).

In 2001, Ruiz used very precise estimates for the resolvent of the Laplacian to prove
that for non-smooth potential the main singularities of the potential (in the scale of Sobolev spaces) are contained  in the fixed angle Born approximation, which is a linear approximation of the potential that we define below (see  \cite{Ru1}).

More recently, Kilgore, Moskow and Schotland studied the convergence and stability of the Born series and its inverse for several inverse scattering problems (see \cite{MS2008,KMS2012,KMS2017}).
They also made numerical studies in \cite{MS2009}. From the numerical point of view,
Barcel\'o, Castro and Reyes studied the recovery of a potential from scattering data 
using a fixed point algorithm which is not justified from the theoretical point of view (see \cite{BCR}).  

All these results are for real potentials, the more general case of complex potentials was treated in \cite{Mochizuki} by Mochizuki and in \cite{BFRV10} by Barcel\'o, Faraco, Ruiz and Vargas.

The aim of this paper is to construct an iterative method for recovering a potential $q$ from fixed angle scattering data.
More precisely, we obtain a new convergent algorithm that combines two approaches, the nonlinear approximation described by Prosser and the fixed point algorithm of Barcel\'o, Castro and Reyes. It is worth mentioning that, even if the numerical approximations based on these approaches work fine, there is no rigorous proofs in the literature supporting any of them. The interest of the new algorithm is that, on one hand, it is computationally faster than the one described in \cite{BCR}, as we show below, and on the other hand, we are able to prove the convergence rigorously. 


To state our main result, we first rewrite the inverse scattering problem in an equivalent integral formulation. Let us define $R_k$ the outgoing resolvent operator of the Laplace operator given, in terms of the Fourier transform, by
\begin{equation}
\label{resolvente}
\widehat{R_k(f)}(\xi)=\frac{\widehat{f}(\xi)}{-|\xi|^2+k^2+i0}.
\end{equation}
Then, $u_s$ is the solution of the so-called Lippmann-Schwinger integral equation
\begin{equation}\label{Lippmann}
u_s(x, \theta_0, k)=R_k(q e^{ik\theta_0\cdot(\cdot)})(x)+R_k( q u_s(\cdot, \theta_0, k))(x), \qquad x\in\mathbb{R}^d.
\end{equation}
Moreover, from \eqref{amplitud} we have that
\begin{equation}\label{amplitud2}
u_\infty (\theta, \theta_0,k)=\int_{\mathbb{R}^d}e^{-ik (\theta-\theta_0) \cdot y}q(y) dy + \int_{\mathbb{R}^d}e^{-ik \theta \cdot y}q(y) u_s(y, \theta_0,k)dy.
\end{equation}
The problem is then to find an approximation of the potential $q$ knowing that it satisfies (\ref{Lippmann}) and (\ref{amplitud2}) parametrized by the scattering data $u_\infty (\theta, \pm \theta_0,k)$ with $(\theta, k) \in S^{d-1} \times {(0, \infty)}$.

If we formally remove the last term in (\ref{amplitud2}), the right hand side can be interpreted as a suitable Fourier transform that can be inverted to obtain the so-called Born approximation. More precisely, given $\theta_0$ fixed, we have, up to a zero measure set, 
\begin{equation}
\label{descomposicion}
\mathbb{R}^d=H_{\theta_0} \cup H_{-\theta_0}=\left\{\xi \in \mathbb{R}^d: \; \xi \cdot \theta_0 <0  \right\}\cup \left\{\xi \in \mathbb{R}^d: \; \xi \cdot \theta_0 >0  \right\}.
\end{equation}
Then, for $\xi \in H_{\pm\theta_0}$, there exists unique $\theta (\xi) \in S^{d-1}$ and $k(\xi)>0$ such that
$$ \xi:= k(\xi)(\theta (\xi)\mp\theta_0)\qquad \text{(see Figura \ref{Ewald})}.$$ 

\begin{figure}[t]
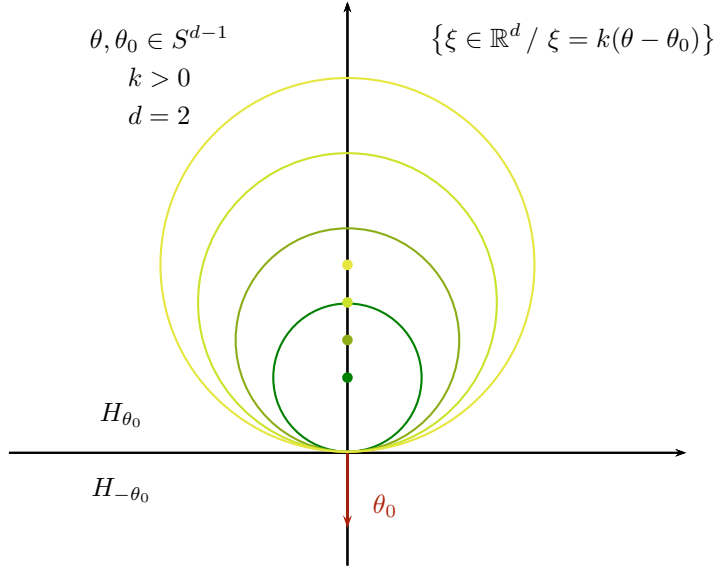

	\rput[t](6.5,0){ \psset{xunit=1cm,yunit=1cm}
		\pspicture[](-4.5,-1.5)(4.5,6)
		\rput(0.5,-0.7){\normalsize{\textcolor[rgb]{0.68,0.13,0.04}{$\theta_0$}}}
		\rput(-2.5,5.5){\normalsize{$\theta,\theta_0\in S^{d-1}$}}
		\rput(-2.5,5){\normalsize{$k>0$}}
		\rput(-2.5,4.5){\normalsize{$d= 2$}}
		\rput(-3,0.5){\normalsize{$\displaystyle H_{\theta_0}$}}
		\rput(-3,-0.5){\normalsize{$\displaystyle H_{-\theta_0}$}}
		\rput(3,5.5){\normalsize{$\displaystyle \left\{\xi\in\mathbb{R}^d\,/\ \xi=k(\theta-\theta_0)\right\}$}}
		\psaxes[linewidth=.7pt,labels=none,ticks=none]{->}(0,0)(-4.5,-1.5)(4.5,6)
		\psline[linecolor=rojo,linewidth=1pt]{->}(0,0)(0,-1)
		\pscircle[linecolor=darkgreen](0,1){1}
		\pscircle[linecolor=darkgreen,fillstyle=solid,fillcolor=darkgreen](0,1){0.07}
		\pscircle[linecolor=verde](0,1.5){1.5}
		\pscircle[linecolor=verde,fillstyle=solid,fillcolor=verde](0,1.5){0.07}
		\pscircle[linecolor=pistacho](0,2){2}
		\pscircle[linecolor=pistacho,fillstyle=solid,fillcolor=pistacho](0,2){0.07}
		\pscircle[linecolor=verdeclaro](0,2.5){2.5}
		\pscircle[linecolor=verdeclaro,fillstyle=solid,fillcolor=verdeclaro](0,2.5){0.07}
		\endpspicture}
	\vspace{7cm}
	\caption{
		Ewald spheres are centered at $-k\theta_0$ with radius $k.$
	}
	\label{Ewald}
\end{figure}

Let us write 
\begin{equation}
\label{theta0}
\theta_0(\xi)=\left\{ \begin{array}{ll}
\theta_0, & \mbox{ if } \xi\in  H_{\theta_0} \\
-\theta_0, & \mbox{ if } \xi\in H_{-\theta_0}
\end{array} \right.
\end{equation}
Then, the Born approximation for fixed angle scattering data $\theta_0$ of a potential $q$ is defined by
\begin{equation}
\label{eq:Born_ap}
\widehat{q_{\theta_0}}(\xi)=u_\infty(\theta(\xi),\theta_0(\xi),k(\xi)). 
\end{equation}
Note that this definition requires scattering data for both $\theta_0$ and $-\theta_0$.

The algorithm proposed in \cite{BCR} used this Born approximation to approximate $u_s(y,\theta_0,k)$ in (\ref{amplitud2}) iteratively. More precisely, a sequence of potentials are defined by $q_1= q_{\theta_0}$ and 
$$
\widehat{q_{n+1}}(\xi)=u_\infty(\theta(\xi),\theta_0(\xi),k(\xi))-\int_{\mathbb{R}^d}e^{-ik(\xi) \theta_0 \cdot y}q_{n}(y) u_s^n(y, \theta_0,k(\xi))dy ,
$$
where $u_s^n$ solves 
\begin{equation} \label{eq_c1}
u_s^n(x, \theta_0, k)=R_k(q_n e^{ik\theta_0\cdot(\cdot)})(x)+R_k( q_n u_s^n(\cdot, \theta_0, k))(x), \qquad x\in\mathbb{R}^d. 
\end{equation}
This requires to solve the Lipmann-Schwinger equation (\ref{Lippmann}) for all $k(\xi)$ at each iteration. As described in \cite{BCR} the numerical version of this algorithm converge in few iterations but each one is expensive computationally, even for 2-d problems. 

To avoid the solution  of equation (\ref{eq_c1}) 
we insert iteratively the Lippmann-Schwinger integral equation (\ref{Lippmann}) into (\ref{amplitud}). In this way we obtain the Born series
\begin{eqnarray}\label{eq:expansion}
u_\infty (\theta, \theta_0,k)&=&\int_{\mathbb{R}^d}e^{-ik (\theta-\theta_0) \cdot y}q(y)dy\nonumber\\
&&+\sum_{j=1}^m\int_{\mathbb{R}^d}e^{-ik \theta \cdot y}(qR_k)^j(q e^{ik\theta_0\cdot(\cdot) })(y)dy\\
&&+\int_{\mathbb{R}^d}e^{-ik \theta \cdot y}(qR_k)^m(q u_s(\cdot,\theta_0,k))(y)dy\nonumber.
\end{eqnarray}
This is the approach followed by R.T. Prosser. At this point, Prosser introduces a classical asymptotic method to recover $q$ based on writing $q$ on power series $q=\sum_{n=0}^\infty \varepsilon^n q_n$, substituting in \eqref{eq_c1} and identifying the terms with the same powers of $\varepsilon.$

Here we follow a different approach.
For convenience, we rewrite \eqref{eq:expansion} as follows
\begin{equation}\label{eq:born_expansion}
\widehat q(\xi)=\widehat{q_{\theta_0}}(\xi)-\sum_{j=1}^m\widehat{ \mathcal{Q}_j(q)}(\xi)-\widehat{ q_m^r}(\xi),\qquad\xi\in\mathbb{R}^d,
\end{equation}
where
\begin{equation}\label{Q_operator}
\widehat{\mathcal{Q}_j(q)}(\xi):=\int_{\mathbb{R}^d}e^{-ik(\xi) \theta(\xi) \cdot y}(qR_{k(\xi)})^j(qe^{ik(\xi)\theta_0\cdot(\cdot) })(y)dy
\end{equation}
and
\begin{equation}\label{resto_m}
\widehat{q_m^r}(\xi):=\int_{\mathbb{R}^d}e^{-ik(\xi) \theta(\xi) \cdot y}(qR_{k(\xi)})^m(qu_s(\cdot,\theta_0,k))(y)dy.
\end{equation}

The convergence of this series suggests that the last term in equation (\ref{eq:born_expansion})  should be small for large $m$. Based on this idea we consider the following family of reduced equations for $q_m$, where we have removed this last term,
\begin{equation}\label{eq:expansion_red}
\widehat {q_m}(\xi)=\widehat{q_{\theta_0}}(\xi)-\sum_{j=1}^m\widehat{ \mathcal{Q}_j(q_m)}(\xi),\qquad\xi=k(\xi)(\theta (\xi)\mp\theta_0),
\end{equation}
with $(\theta, k) \in S^{d-1} \times (0, \infty).$
These reduced equations have the advantage that they do not involve $u_s$, avoiding the solution of the Lipmann-Schwinger equation (\ref{Lippmann}). However, for each parameter $(\theta, k) \in  S^{d-1} \times (0, \infty)$, equation (\ref{eq:expansion_red}) is still nonlinear in $q_m$. Moreover, it is not clear if there exists a unique function $q_m$ satisfying (\ref{eq:expansion_red}).

Here we propose a fixed point procedure to find approximations of $q_m$. More precisely we introduce the linear operator $\mathcal{L}_m$ defined by
$$
\widehat{\mathcal{L}_m (q)} (\xi):= \widehat{q_{\theta_0}}(\xi)- \sum_{j=1}^m \widehat{\mathcal{Q}_j(q)}(\xi), 
$$

Then, if $q_m$ is solution of (\ref{eq:expansion_red}), it must be also a fixed point of $\mathcal{L}_m$ and we can try the usual iterative method based on powers of $\mathcal{L}_m$ to approximate $q_m$. This requires in particular that $\mathcal{L}_m(q)$ is of compact support. Therefore, instead of $\mathcal{L}_m$ we consider the modified operator 
\begin{equation}\label{movil3}
\mathcal{T}_m(q):=  
\phi \mathcal{L}_m(q)
=
\phi q_{\theta_0}
-\phi \sum_{j=1}^m \mathcal{Q}_j(q),
\end{equation}
where $\phi\in\mathcal{C}^\infty$ is a cut-off function with compact support
satisfying
\begin{equation}\label{movil}
\phi(x)=1 , \textrm{ if} \; |x|<R \;\, \textrm{ and } \; \phi (x)=0, \; \textrm{ if } |x|>2R.
\end{equation}

We are now ready to state the main results in this paper.
For each $m\in\mathbb{N}$, we consider the sequence $\{q_{m,\ell}\}_{\ell \in \mathbb{N}}$  defined recursively by
\begin{equation}\label{def:sequence}
\left\{
\begin{array}{ll}
q_{m,1}=0,\\
q_{m,\ell+1}=\mathcal{T}_m(q_{m,\ell}),\qquad \ell\ge 1.
\end{array}
\right.
\end{equation}
We note that $q_{m,2}=\phi q_{\theta_0},$ which is a good approximation to a potential $q$ with support in $B(0,R)$. 

The purpose of this paper is to prove that the sequence of approximations  $\{q_{m,\ell}\}_{m,\ell \in \mathbb{N}}$ converges to the potential $q$ in some sense.
More precisely, we will prove the following theorem.
\begin{theorem}\label{th:cauchy_sequence} For $d\ge 2$ and $\alpha$ satisfying 
	\begin{equation}
	\label{st}
	0<\alpha\le 1,\quad\textrm{ and }\quad\frac{d}{2}-\frac{d}{d-1}<\alpha<\frac{d}{2},
	\end{equation}
	let $q\in W^{\alpha,2}(\mathbb{R}^d)$ be a real valued function with compact support in $B(0,R)$ and such that 
	\begin{equation}
	\label{pequet}
	\|q\|_{W^{\alpha,2}}< A,
	\end{equation} 
	for an appropriate constant $A:=A(d,\alpha,R)>0$ small enough (see Remark \ref{pequenez}).
	For each $m\in\mathbb{N}$, let $\{q_{m,\ell}\}_{\ell\in\mathbb{N}}$ be the sequence defined by \eqref{def:sequence}. 
	Then, there exists $q_m\in W^{\alpha,2}(\mathbb{R}^d)$ satisfying
	$$
	q_m=\lim_{\ell\rightarrow\infty}q_{m,\ell}\qquad \text{in }W^{\alpha,2}(\mathbb{R}^d).
	$$
	Moreover, the sequence $\{q_m\}_{m\in\mathbb{N}}$ satisfies
	$$
	\lim_{m\rightarrow\infty}q_m=q\qquad \text{in }W^{\alpha,2}(\mathbb{R}^d).
	$$
\end{theorem}
\begin{remark}
	\label{pequenez}
	Following the proof of Theorem \ref{th:cauchy_sequence}, one can see that the smallness condition given in \eqref{pequet} is 
	$$
	\|q\|_{W^{\alpha,2}}< A=\min\left(\frac{1}{C_1},\frac{1}{C_2},\frac{1}{2C_1C_3C_4},\frac{1}{2^{2+\frac{d-1}{2}\left(\frac{1}{2}-\frac{\alpha}{d}\right)}C_5}\right),
	$$
	where $C_1,C_2,C_3,C_4$ and $C_5$ are the constants that appear in 
	\eqref{estimacionPj}, \eqref{eq:estima_resto}, \eqref{control_aproximacion_born}, \eqref{rey_3} and \eqref{lema1} respectively (see below).
\end{remark}
\begin{remark}
	\label{rango}
	Conditions in \eqref{st} imply that $d<5$, and therefore Theorem \ref{th:cauchy_sequence} is valid for $2\le d\le 4.$ Moreover,
	for $d=2,$ the result holds if $0<\alpha<1,$
	for $d=3,$ if $0<\alpha\le 1,$ and
	for $d=4,$ if $2/3<\alpha\le 1.$  
\end{remark}
\begin{remark}
	\label{velocidad}
	In the proof of Theorem \ref{th:cauchy_sequence} (see \eqref{fin} below) we will see that the convergence of the sequence $\{q_m\}_{m\in\mathbb{N}}$ to potential $q$ is faster than $(C_2\|q\|_{W^{\alpha,2}})^{m}$, where $C_2$ is the constant that appears in \eqref{eq:estima_resto} (see below).
\end{remark}
\paragraph{Notation.} 
For $\alpha\in\mathbb{R}$ we introduce the fractional differentiation operator
$$
\Lambda^{\alpha}
=
(1+\Delta)^{\alpha/2}
=
\mathcal{F}^{-1}\left\langle \xi\right\rangle^{\alpha}\mathcal{F},
$$
where $\mathcal{F}$ denotes the Fourier transform and
$\left\langle\xi\right\rangle =(1+\left\vert \xi\right\vert ^{2})^{1/2}.$

We use the Sobolev spaces
\begin{equation*}
\label{sobolev}
W^{\alpha,p}(\mathbb{R}^d)=\{f\in\mathcal{S}^{\prime}(\mathbb{R}^{d}):\Lambda^{\alpha}f\in L^{p}(\mathbb{R}^d)\},\qquad \alpha\in\mathbb{R},1\le p\le\infty,
\end{equation*}
and also their weighted versions
$$
W_{\delta}^{\alpha,p}(\mathbb{R}^d)
=
\{f\in\mathcal{S}^{\prime}(\mathbb{R}^{d}):\Lambda^{\alpha}f\in L_{\delta}^{p}(\mathbb{R}^d)\},
\qquad \alpha,\delta\in\mathbb{R},1\le p\le\infty,
$$
where $L_{\delta}^{p}(\mathbb{R}^d)=\{f:\left\langle x\right\rangle ^{\delta}f\in L^{p}(\mathbb{R}^d)\}.$

Throughout this paper $C$ will denote a positive constant that may change from line to line and depend on some parameters such as $d,\alpha$ or $R$. This dependence will be indicated when relevant.
In order to know in detail the required smallness condition on the potential $q,$ in some cases, we will name constants by $C_n$ with $n$ varying in $\mathbb{N}.$ Moreover we will write $C_{n_1,\ldots,n_m}$ to indicate the product $C_{n_1}\ldots C_{n_m},$ $n_1,\ldots n_m\in\mathbb{N}.$

The rest of this paper is organized as follows.
The proof of Theorem \ref{th:cauchy_sequence} is given in the second section and also several lemmas needed in the proof. 
In the third section we illustrate these results with several numerical experiments. 
\section{Proofs}
This section is devoted to the proof of Theorem \ref{th:cauchy_sequence}.
We split it into two subsections. The first one contains the proof itself and the statement of two results (see Propositions \ref{proposicion_apendice} and \ref{resto} below) which are the key points in the proof. The proofs of these propositions are quite technical, and require of several known results, so we postpone them to the second subsection.
\subsection{Proof of Theorem \ref{th:cauchy_sequence}}
The first key point in the proof of Theorem \ref{th:cauchy_sequence} are some estimates for a family of operators that generalize the operators $\mathcal{Q}_j$ given in \eqref{Q_operator}.

For every $j\in\mathbb{N}$, we introduce the following multilinear operator defined via its Fourier transform as follows
\begin{equation}\label{eq:decomposition}
\mathcal{P}_j(\mathbf{f})(x)
=\int_{\mathbb{R}^d}e^{ix\cdot\xi}\widehat{\mathcal{P}_j(\mathbf{f})}(\xi)d\xi
=\left(\int_{H_{\theta_0}}+\int_{H_{-\theta_0}}\right)e^{ix\cdot\xi}\widehat{\mathcal{P}_j(\mathbf{f})}(\xi)d\xi,
\end{equation}
where $\mathbf{f}=(f_1,f_2,\ldots,f_{j+1}),$ 
$\theta_0 \in S^{d-1}$  and
$H_{\pm\theta_0}$ is given in \eqref{descomposicion}.

Taking into account that any $\xi\in H_{\pm\theta_0}$ can be written in a unique way as $\xi=k(\theta\mp\theta_0)$, with $k=k(\xi)>0$ and $\theta=\theta(\xi)\in S^{d-1}$, we define
\begin{equation}\label{eq:Pj}
\widehat{\mathcal{P}_j(\mathbf{f})}(\xi):=\int_{\mathbb{R}^d}e^{- ik \theta  \cdot y}(f_{j+1}R_{k}\ldots f_2R_{k})(f_1e^{ ik\theta_0\cdot(\cdot) })(y)dy,
\end{equation}
Observe that in the particular case $f_i=q, \; i=1,2, \cdot  \cdot \cdot ,j+1$, we have that
\begin{equation}\label{rey_1}
\widehat{\mathcal{P}_j(\mathbf{f})}(\xi)=\widehat{\mathcal{Q}_j(q)}(\xi),\qquad\xi\in\mathbb{R}^d.
\end{equation}
\begin{proposition}\label{proposicion_apendice} Let $d\ge 2$ and $\alpha$ satisfying \eqref{st}.
For each $j\in\mathbb{N}$ fixed, let $\mathbf{f}=(f_1,f_2,\cdots f_{j+1})$ with $f_{\ell}\in W^{\alpha,2}(\mathbb{R}^d)$ and compactly supported with support in B(0,R), for $\ell=1,\ldots j+1$. Then, there exits a constant $C_1:=C_1(d,\alpha ,R)$ such that
\begin{equation}\label{estimacionPj}
\left\|  \mathcal{P}_j(\mathbf{f}) \right\|_{W^{\alpha ,2} } \leq  C_1^j   \prod_{\ell =1}^{j+1}
     \|f_{\ell}\|_{W^{\alpha ,2}} .
\end{equation}
\end{proposition}
The other key point in the proof of Theorem \ref{th:cauchy_sequence} concerns with the error term $q_m^r$ defined in \eqref{resto_m}.
\begin{proposition}\label{resto}
For $d\ge 2$ and $\alpha$ satisfying 
\begin{equation}
\label{sqrm}
0<\alpha\le 1,\quad\textrm{ and }\quad\frac{d}{2}-2<\alpha<\frac{d}{2},
\end{equation}
let $q\in W^{\alpha,2}(\mathbb{R}^d)$ be a real valued function with compact support in $B(0,R)$ and such that 
\begin{equation}
\label{peque}
\|q\|_{W^{\alpha,2}}< \frac{1}{2^{2+\frac{d-1}{2}\left(\frac{1}{2}-\frac{\alpha}{d}\right)}C_5},
\end{equation} 
where $C_5$ is the constant that appears in \eqref{lema1}.
Then, for every $m\in\mathbb{N}$ there exists a constant $C_2:=C_2(d,\alpha ,R)$ such that 
\begin{equation}\label{eq:estima_resto}
\|q_m^r\|_{W^{\alpha,2}}\leq C_2^m\|q\|_{W^{\alpha,2}}^{m+1}.
\end{equation}
\end{proposition} 
We also need to control the Born approximation defined in \eqref{eq:Born_ap}.
\begin{corolario}
	\label{Cborn}
	For $d\ge 2$ and $\alpha$ satisfying \eqref{st},
	let $q\in W^{\alpha,2}(\mathbb{R}^d)$ be a real valued function with compact support in $B(0,R)$ satisfying \eqref{peque}.
	Then, there exists a constant $C_3:=C_3(d,\alpha ,R)$ such that 
\begin{equation}\label{control_aproximacion_born}
\| q_{\theta_0}\|_{W^{\alpha,2} } <  C_3 \|q\|_{W^{\alpha , 2} }.
\end{equation}
\end{corolario}
\begin{proof}
From \eqref{eq:born_expansion} with $m=1,$ we have that
$$\| q_{\theta_0}\|_{W^{\alpha,2}}\le  \|q\|_{W^{\alpha , 2}}+\left\| \mathcal{Q}_1(q) \right\|_{W^{\alpha , 2}}+ \left\|q_1^r  \right\|_{W^{\alpha , 2}}.$$
The result follows from here using \eqref{rey_1} and \eqref{estimacionPj} for $j=1,$ \eqref{eq:estima_resto} for $m=1,$ and \eqref{peque}.
\end{proof}
Finally, we need the following result concerning the product of functions in Sobolev spaces due to Zolesio (see \cite{Z}). 
\begin{lemma}
	\label{ZolesioParticular}
	Let $0\le \alpha\le s$ and $s>d/2,$ and
	let $\phi$ be the cut-off function defined in \eqref{movil}, then there exists a constant $C_4:=C_4(d,\alpha ,R)>0$ such that
	\begin{equation}\label{rey_3}
	\|\phi g \|_{W^{\alpha,2}} \leq C \| \phi \|_{W^{s,2}} \| g \|_{W^{\alpha,2}} \leq C_4 \| g \|_{W^{\alpha,2}}.
	\end{equation}
\end{lemma}

\emph{Proof of Theorem \ref{th:cauchy_sequence}.}
We split the proof in three steps. In the first one we will prove that for each $m\in\mathbb{N},$ the sequence $\left\{ q_{m,\ell} \right\}_{\ell \in \mathbb{N}}$  is bounded in the space $W^{\alpha, 2}(\mathbb{R}^d).$ In the second one we will see that  such sequence is a Cauchy sequence in that space, thus it converges to a function $q_m.$ Finally, in the third one we will prove that $q$ is the limit in $W^{\alpha, 2}(\mathbb{R}^d)$ of $q_m$ as $m$ goes to infinity.

\underline{STEP 1}. 
We will prove the boundedness by induction on $\ell.$ Moreover, for each $m\in\mathbb{N},$ we will prove that
\begin{equation}\label{eq:2estimaVn}
\|q_{m,\ell}\|_{W^{\alpha,2}}\leq 2C_{3,4}\| q\|_{W^{\alpha,2}},\qquad \forall\,\ell\ge 2,
\end{equation}
whenever $\alpha$ satisfies \eqref{st} and $q\in W^{\alpha,2}(\mathbb{R}^d)$ is a real valued function with compact support in $B(0,R)$ satisfying \eqref{peque} and such that 
\begin{equation}
\label{peque1}
\|q\|_{W^{\alpha,2}}\le \frac{1}{2C_{1,3,4}(1+2C_4)}.
\end{equation} 

From \eqref{def:sequence} and \eqref{movil3}, we have that
$$
\|q_{m,2}\|_{W^{\alpha,2}}=\|\phi q_{\theta_0}\|_{W^{\alpha,2}}.
$$
Using \eqref{rey_3} and \eqref{control_aproximacion_born} we get
\begin{equation*}
\label{l=2}
\|q_{m,2}\|_{W^{\alpha,2}}\le C_{3,4} \|q\|_{W^{\alpha , 2} }\le 2C_{3,4} \|q\|_{W^{\alpha , 2} },
\end{equation*}
for $\alpha$ and $q$ under the assumptions of Corollary \ref{Cborn}.

Arguing as before, from \eqref{def:sequence}, \eqref{movil3}, \eqref{rey_3} and \eqref{control_aproximacion_born} we also get
\begin{eqnarray*}
\|q_{m,\ell+1}\|_{W^{\alpha,2}}
&\le&
\|\phi q_{\theta_0}\|_{W^{\alpha,2}}
+
\left\| \phi\sum_{j=1}^m \mathcal{Q}_j(q_{m,\ell})  \right\|_{W^{\alpha,2}} 
\nonumber
\\
&\le&
C_{3,4}\| q\|_{W^{\alpha,2}}
+
C_4\sum_{j=1}^m\left\| \mathcal{Q}_j(q_{m,\ell})  \right\|_{W^{\alpha,2}}.
\end{eqnarray*}
From here, using identity \eqref{rey_1}, \eqref{estimacionPj} and the induction hypothesis \eqref{eq:2estimaVn}, we obtain
\begin{eqnarray*}
\|q_{m,\ell+1}\|_{W^{\alpha,2}}
&\le&
C_{3,4}\| q\|_{W^{\alpha,2}}
+
C_4\sum_{j=1}^mC_1^j\left\| q_{m,\ell}  \right\|^{j+1}_{W^{\alpha,2}}
\\
&\le&
C_{3,4}\| q\|_{W^{\alpha,2}}
 +
C_4 \sum_{j=1}^mC_1^j(2C_{3,4})^{j+1}\left\| q \right\|^{j+1}_{W^{\alpha,2}}.
\end{eqnarray*}
Estimate \eqref{eq:2estimaVn} follows from here since 
for $q$ satisfying \eqref{peque1} we have that
$$2C_4\sum_{j=1}^m(2C_{1,3,4}\| q\|_{W^{\alpha,2}})^{j}\le 1.$$

\underline{STEP 2}. 
From \eqref{def:sequence} and \eqref{movil3}, using \eqref{rey_3}, we have that
$$
\|q_{m,\ell+1}-q_{m,n+1}\|_{W^{\alpha,2}}
\le
C_4
\left\| \sum_{j=1}^m \left(\mathcal{Q}_j(q_{m,\ell})- \mathcal{Q}_j(q_{m,n})\right) \right\|_{W^{\alpha,2}}.
$$
From here, using \eqref{rey_1}, the fact that $\mathcal{P}_j$ is a multilinear operator and the triangular inequality, we get
\begin{equation}
\label{inter}
\|q_{m,\ell+1}-q_{m,n+1}\|_{W^{\alpha,2}}
\le
C_4
\sum_{j=1}^m\sum_{i=1}^{j+1}\left\| \mathcal{P}_j(\mathbf{f}_{m,\ell,n,i}) \right\|_{W^{\alpha,2}},
\end{equation}
where
$$
\mathbf{f}_{m,\ell,n,i}=(q_{m,\ell},\ldots q_{m,\ell},\underbrace{q_{m,\ell}-q_{m,n}}_{i-position},q_{m,n},\ldots,q_{m,n}).
$$
Using \eqref{estimacionPj} and \eqref{eq:2estimaVn} in \eqref{inter} we obtain
\begin{eqnarray}
\|q_{m,\ell+1}-q_{m,n+1}\|_{W^{\alpha,2}}
&\le&
C_4
\sum_{j=1}^m
(j+1)
\left(2C_{1,3,4}\left\|q\right\|_{W^{\alpha,2}}\right)^j
\left\| q_{m,\ell}-q_{m,n}\right\|_{W^{\alpha,2}}
\nonumber
\\
&\le&
B
\left\| q_{m,\ell}-q_{m,n}\right\|_{W^{\alpha,2}},
\label{cauchy}
\end{eqnarray}
with
$$
B=\frac{C_4}{(1-2C_{1,3,4}\left\|q\right\|_{W^{\alpha,2}})^2}.
$$
Observe that $B<1$  whenever
$\alpha$ satisfies \eqref{st} and $q\in W^{\alpha,2}(\mathbb{R}^d)$ is a real valued function with compact support in $B(0,R)$ satisfying \eqref{peque}, \eqref{peque1} and
\begin{equation}
\label{peque2}
\left\|q\right\|_{W^{\alpha,2}}<\frac{1+\sqrt{C_4}}{2C_{1,3,4}}.
\end{equation}
Therefore, for such a $q,$ we have that $\left\{ q_{m,\ell} \right\}_{\ell \in \mathbb{N}}$ is a Cauchy sequence.
As a consequence, there exits $q_m \in W^{\alpha, 2}(\mathbb{R}^d)$ such that
$$q_m= \lim_{\ell \rightarrow \infty}q_{m,\ell},\qquad\text{in }W^{\alpha, 2}(\mathbb{R}^d).$$
From \eqref{eq:2estimaVn}, we have that $q_m$ satisfies 
\begin{equation}\label{acotaVm}
\|q_m\|_{W^{\alpha,2}}\leq 2C_{3,4}\| q\|_{W^{\alpha,2}}.
\end{equation}
Moreover,
\begin{equation}\label{movil5}
q_m= \phi  q_{\theta_0}-\phi \sum_{j=1}^m \mathcal{Q}_j(q_m),
\end{equation}
since arguing as we did to get \eqref{cauchy}, we obtain
\begin{equation*}
\left\|q_{m,\ell+1}-\left(\phi  q_{\theta_0}-\phi \sum_{j=1}^m \mathcal{Q}_j(q_m)\right)\right\|_{W^{\alpha,2}}
<
\left\| q_{m,\ell}-q_m\right\|_{W^{\alpha,2}}.
\end{equation*}

\underline{STEP 3}. 
For $\alpha$ satisfying \eqref{st} and $q\in W^{\alpha,2}(\mathbb{R}^d)$ being a real valued function with compact support in $B(0,R)$ satisfying \eqref{peque}, \eqref{peque1} and \eqref{peque2},
from \eqref{movil5}, \eqref{eq:born_expansion} and \eqref{movil}, since $supp(q)\subset B(0,R),$ we have that
$$
\|q_m-q\|_{W^{\alpha,2}}
\le
\left\|\phi \sum_{j=1}^m \left(\mathcal{Q}_j(q_m)- \mathcal{Q}_j(q)\right) \right\|_{W^{\alpha,2}}
+
\| \phi q_m^r\|_{W^{\alpha,2}}.
$$
From here, arguing as we did to get \eqref{cauchy}, we obtain
\begin{equation}
\label{final}
\|q_m-q\|_{W^{\alpha,2}}
\le
D
\|q_m-q\|_{W^{\alpha,2}}
+
\| \phi q_m^r\|_{W^{\alpha,2}},
\end{equation}
where 
$$
D=
C_4
\sum_{j=1}^m
C_1^j
\sum_{i=1}^{j+1}
\| q_m\|_{W^{\alpha,2}}^{i-1}
\|q\|_{W^{\alpha,2}}^{j+1-i}.
$$
Using \eqref{acotaVm}, we have that
\begin{equation*}
\label{D}
D\le 
C_4
\sum_{j=1}^m
(C_1\|q\|_{W^{\alpha,2}})^j
\sum_{i=1}^{j+1}
(2C_{3,4})^{i-1},
\end{equation*}
and therefore
\begin{equation*}
D\le 
\left\{
\begin{array}{lll}
\displaystyle
C_4
\sum_{j=1}^m
(j+1)
(C_{1}\|q\|_{W^{\alpha,2}})^j,
&&\text{if } 2C_{1,3,4}\le1,
\\[2ex]
\displaystyle
C_4
\sum_{j=1}^m
(j+1)
(2C_{1,3,4}\|q\|_{W^{\alpha,2}})^j,
&&\text{if } 2C_{1,3,4}\ge1.
\end{array}
\right.
\end{equation*}
Thus,
$D\le 1/2$ 
whenever
\begin{equation}
\label{peque3}
\left\|q\right\|_{W^{\alpha,2}}<\min\left(\frac{1}{C_1},\frac{1}{2C_{1,3,4}}\right).
\end{equation}
Using this in \eqref{final} we get
\begin{equation*}
\|q_m-q\|_{W^{\alpha,2}}
\le
2
\| \phi q_m^r\|_{W^{\alpha,2}}.
\end{equation*}
Finally, using \eqref{rey_3} and \eqref{eq:estima_resto} we have that
\begin{equation}
\label{fin}
\|q_m-q\|_{W^{\alpha,2}}
\le
2C_4C_2^m\|q\|_{W^{\alpha,2}}^{m+1}.
\end{equation}
The result follows from here if $q$ satisfies  \eqref{peque}, \eqref{peque1}, \eqref{peque2}, \eqref{peque3}, and also
$
\|q\|_{W^{\alpha,2}}<1/C_2.
$
\hfill
$\Box$
\subsection{Proofs of  the key points}
In this subsection we give the proofs of Propositions \ref{proposicion_apendice} and \ref{resto}. They require of several known estimates for different operators in weighted Sobolev spaces.

We will need to know the behaviour of the outgoing resolvent of the laplacian 
denoted by $R_k$ (see \eqref{resolvente}) for $k$ small given in the following lemma.
\begin{lemma}\label{le:Rkpequeño}{\rm (\cite[Chapter VI]{KK},  \cite[lemma 21.4]{Ko}.).} Let $\alpha \in [0,2]$, $k\in[0,b],$ $b>0$ and $\delta>1$. Then
	\begin{equation}\label{zapatero_1}
	\|R_kf\|_{W^{\alpha,2}_{-\delta}} \leq C_{b,\delta}\|f\|_{L^2_\delta }.
	\end{equation}
\end{lemma}

For $\theta_0\in S^{d-1}$ and $k>0$ fixed, we introduce the following operator involving the outgoing resolvent of the laplacian,
\begin{equation}
\label{Rconjugado}
R_{k, \theta_0}f(x)=e^{-ik \theta_0 \cdot x}R_k(f(\cdot )e^{ik \theta_0 \cdot (\cdot)})(x).
\end{equation}
The following result can be obtained by interpolation of several estimates due to Agmon (see \cite{Ag}), Kenig, Ruiz and Sogge (see \cite{KRS}), Ruiz and Vega (see\cite{RV}), and Barcel\'o, Ruiz and Vega (see \cite{BRV}). For details we refer the reader to \cite{Ru1}.	
\begin{lemma} {\rm(\cite[Lemma 3.4]{Ru1})}\label{le:Rk_grande}. Let $\alpha\geq 0$, $r$ and $t$ such that $0 \leq \frac{1}{t}-\frac{1}{2} \leq \frac{1}{d+1}$ and $0 \leq \frac{1}{2}-\frac{1}{r} \leq \frac{1}{d+1}$, then there exist  $\delta>1$ and a constant $C_5:=C_5(\delta)$ such that
	\begin{equation}\label{lema1}
	\|R_{k, \theta_0}f \|_{W^{\alpha,r}_{-\delta}} \leq C_5 k^{-1+\frac{d-1}{2}\left(\frac{1}{t}-\frac{1}{r}  \right)}\|f \|_{W^{\alpha,t}_{\delta}} .
	\end{equation}
\end{lemma}
We also introduce the restriction operator given by
\begin{equation}\label{def:restriction}
S_{k,\theta_0}f(\theta)=\int_{\mathbb{R}^d}e^{-ik(\theta-\theta_0) \cdot y}f(y) dy,\qquad \theta\in S^{d-1}.
\end{equation}
The following lemma is a consequence of Theorem 3(c) in \cite{BRV}. 
\begin{lemma}
	\label{LemaRayosX}
	If $\delta>1/2$ then, there exists $C>0$ such that 
	\begin{equation}\label{rayosX}
	\|S_{k,\theta_0}f \|_{L^2(S^{d-1})} \leq Ck^{-\frac{d-1}{2}}\|f\|_{L^2_{\delta}}.
	\end{equation}
\end{lemma}
Using the previous lemma and the Stein-Tomas restriction theorem(see \cite{SteinTomas}) we can get the following result which generalizes the previous one.
\begin{lemma}\label{le:Sk} {\rm(\cite[Lemma 3.7]{Ru1}).} Let $\alpha\geq 0$ and $t$ satisfying $0 \leq \frac{1}{t}-\frac{1}{2}\leq \frac{1}{d+1}$; then there exists $\delta (t) >0$ and $C:=C_t>0$ such that for $w(\theta)=|\theta-\theta_0|^\alpha$, we have 
	\begin{equation}\label{lema2}
	\|wS_{k,\theta_0}f \|_{L^2(S^{d-1})} \leq Ck^{\frac{d-1}{2}\left( \frac{1}{t}-\frac{3}{2}   \right)-\alpha}\|f\|_{W^{\alpha,t}_{\delta (t)}}.
	\end{equation}
\end{lemma}

Now we state the following result concerning the product of functions in weighted Sobolev spaces due to Zolesio (see \cite{Z}). For a more general version and details we refer the reader to Theorem 1.4.4.2 (pp. 28) of \cite{Grisvard}.
See also Proposition D.3 (pp. 182) of \cite{TesisReyes} for functions compactly supported. 
\begin{lemma}\label{le:zolesioNecesario}
	Let $f$ be compactly supported and $\delta \in \mathbb{R}$.
	\begin{itemize}
		\item [(i)] For $\alpha$ satisfying
		\begin{equation}
		\label{restriccionesI1a}
		0\le \alpha\quad\text{and}\quad \alpha> \frac{d}{2}-2,
		\end{equation}
		or
		\begin{equation}
		\label{restriccionesI1b}
		0< \alpha\quad\text{and}\quad \alpha\ge \frac{d}{2}-2,
		\end{equation}
		we have that there exists a constant $C:=C(\supp (f))$ such that
		\begin{equation}\label{zolesioI1}
		\|fg\|_{L^2_\delta} \leq C\|f\|_{W^{\alpha,2}} \|g\|_{W^{2,2}_{-\delta}} .
		\end{equation}
		\item [(ii)] For $\alpha,t$ and $r$ satisfying
		\begin{equation*}
		1\le t < \min(2, r)\quad \text{   and   } \quad 0 \leq\frac{1}{2}+\frac{1}{r}-\frac{1}{t} \leq \frac{\alpha}{d},
		\end{equation*}	
		or
		\begin{equation*}
		1\le t \le \min(2, r)\quad \text{   and   } \quad 0 \leq\frac{1}{2}+\frac{1}{r}-\frac{1}{t} < \frac{\alpha}{d},
		\end{equation*}		
		we have that there exists a constant $C:=C(\supp (f))$ such that
		\begin{equation}\label{zolesioI2}
		\|fg\|_{W^{\alpha,t}_{\delta}} \leq C\|f\|_{W^{\alpha,2}} \|g\|_{W^{\alpha,r}_{-\delta}}.
		\end{equation}
		\item [(iii)] For $\alpha,\beta$ and $p$ satisfying
		\begin{equation*}
		\label{restriccionesZ3}
		0\le\beta\le\min(\alpha,2),\quad 1\le p<2\quad \text{   and   } \quad
		0\le d\left(1-\frac{1}{p}\right)\le \alpha+2-\beta,
		\end{equation*}
		we have that there exists a constant $C:=C(\supp (f))$ such that
		\begin{equation}\label{Z3}
		\|fg\|_{W^{\beta,p}} \leq C\|f\|_{W^{\alpha,2}} \|g\|_{W^{2,2}_{-\delta}}.
		\end{equation}
	\end{itemize}
\end{lemma}

\emph{Proof of Proposition \ref{proposicion_apendice}.}
From \eqref{descomposicion}, for a fixed $\theta_0\in S^{d-1}$ 
we have that
\begin{equation*}
\left\|  \mathcal{P}_j(\mathbf{f}) \right\|_{W^{\alpha ,2}}^2
=
\left(
\int_{H_{\theta_0}}
+
\int_{H_{-\theta_0}}
\right)
\left\langle\xi\right\rangle^{2\alpha} \left|\widehat{\mathcal{P}_j(\mathbf{f})}(\xi)\right|^2 d \xi
=
I+II.
\end{equation*}
We will prove the estimate for term $I,$ the other one is obtained in a similar way.

Observe that in $H_{\theta_0}$ we can make the following change of variables
\begin{equation}
\label{cambio}
\xi=k(\theta - \theta_0), \hspace{0.5cm} d \xi= k^{d-1}|\theta - \theta_0|^2d\sigma(\theta) dk.
\end{equation}
Then,
\begin{eqnarray}
I
&\le&
C\int_0^1k^{d-1}\int_{S^{d-1}}\left| \widehat{\mathcal{P}_j(\mathbf{f})}(k(\theta-\theta_0))\right|^2d\sigma (\theta)dk
\nonumber
\\
&&
+C\int_1^\infty k^{d-1+2\alpha}\int_{S^{d-1}}|\theta-\theta_0|^{2\alpha}\left|\widehat{\mathcal{P}_j(\bf{f})}(k(\theta-\theta_0))\right|^2d\sigma (\theta)dk
\nonumber
\\
&=&
C(I_1+I_2),
\label{I}
\end{eqnarray}
whenever $0\le\alpha\le 1.$

In order to control $I_1,$ we rewrite $\mathcal{P}_j$ in terms of the resolvent operator $R_k$ and the restriction operator $S_{k,\theta_0}$ defined in \eqref{def:restriction}, as follows:
\begin{equation*}\label{eq:PviaSR}
\widehat{\mathcal{P}_j\left(\textbf{f} \right)}(k(\theta-\theta_0))=S_{k,\theta_0}\left(e^{-ik\theta_0 \cdot (\cdot)} f_{j+1}R_k \cdot \cdot \cdot f_2R_k \left( f_1e^{ik\theta_0 \cdot (\cdot)}  \right)    \right) (\theta).
\end{equation*}
Using this identity we have that
$$I_1 = \int_0^1 k^{d-1 }\left\|  S_{k,\theta_0}\left(e^{-ik \theta_0 \cdot (\cdot)} f_{j+1}R_k \cdot \cdot \cdot f_2R_k \left( f_1e^{ik\theta_0 \cdot (\cdot)}  \right)    \right)    \right\|^2_{L^2(S^{d-1})}dk.$$
We can bound the $L^2(S^{d-1})$-norm that appears in the last identity using \eqref{rayosX},  \eqref{zolesioI1} and \eqref{zapatero_1} as the following diagram illustrates
\begin{equation*}
\begin{CD}
L^2_\delta
\ldots
@>
{j-{\texttt{\rm{times}}}}
>>
L^2_\delta
@>
R_k
>>
W_{-\delta}^{2,2}
@>
{\texttt{\rm{Zolesio}}}
>>
L^2_\delta
@>
S_{k,\theta_0}
>>
L^2(S^{d-1}).
\end{CD}
\end{equation*}
Therefore, if  $\delta>1$ and $\alpha$ satisfies \eqref{restriccionesI1a} or \eqref{restriccionesI1b}
we get
\begin{eqnarray}
I_1 &\le& C^{2j }\|f_{j+1}\|^2_{W^{\alpha, 2}}    \|f_{j}\|^2_{W^{\alpha , 2}} \cdot \cdot \cdot  
\left\| f_{2} \right\|^2_{W^{\alpha, 2}}\left\| f_{1} \right\|^2_{L^{ 2}_\delta}
\nonumber
\\
&\le&  C^{2j} \prod_{\ell =1}^{j+1}   \|f_{\ell}\|^2_{W^{\alpha, 2}}.
\label{I1}
\end{eqnarray}

In order to control $I_2$ we need to take more advantage of oscillations, so we write $\mathcal{P}_j$ in terms of the operator $R_{k,\theta}$ introduced in \eqref{Rconjugado} and the restriction operator $S_{k,\theta_0}$ defined in \eqref{def:restriction}, as follows:
\begin{equation*}\label{eq:PviaRconjugado}
\widehat{\mathcal{P}_j\left(\textbf{f} \right)}(k(\theta-\theta_0))
=S_{k,\theta_0}\left(f_{j+1}R_{k,\theta_0} \cdot \cdot \cdot f_2R_{k,\theta_0} f_1  \right) (\theta).
\end{equation*}
Using this identity, and writing $w=|\theta-\theta_0|^\alpha,$ we have that
\begin{equation*}\label{integral_4}
I_2 = \int_1^\infty k^{d-1+2 \alpha}\left\| w S_{k,\theta_0}\left(f_{j+1}R_{k,\theta_0} \cdot \cdot \cdot f_2R_{k,\theta_0} f_1  \right)    \right\|^2_{L^2(S^{d-1})}dk,
\end{equation*}
In this case, we can bound the $L^2(S^{d-1})$-norm that appears in the last identity using \eqref{lema2}, \eqref{zolesioI2} and \eqref{lema1} as the following diagram illustrates
\begin{equation*}
\begin{CD}
W^{\alpha, 2}_\delta
\ldots
@>
{(j-1)-{\texttt{\rm{times}}}}
>>
W^{\alpha, t_j}_{\delta}
@>
R_{k,\theta_0}
>>
W^{\alpha,r_j}_{-\delta}
@>
{\texttt{\rm{Zolesio}}}
>>
W^{\alpha,t_{j+1}}_{\delta} 
@>
S_{k,\theta_0}
>>
L^2(S^{d-1}),
\end{CD}
\end{equation*}
whenever  there exist $r_{\ell}$ and $t_{\ell+1}$ satisfying for $\ell=1,2\ldots j$
\begin{equation}\label{monedero_3}
0 \leq \frac{1}{t_{\ell+1}}- \frac{1}{2} \leq \frac{1}{d+1}\quad\text{  and  } \quad  0 \leq \frac{1}{2}- \frac{1}{r_\ell} \leq \frac{1}{d+1}, 
\end{equation}
\begin{equation}\label{monedero_4}
t_{\ell+1} < \min(2, r_\ell)\quad \text{   and   } \quad 0 \leq\frac{1}{2}+\frac{1}{r_\ell}-\frac{1}{t_{\ell+1}} \leq \frac{\alpha}{d}.
\end{equation}  
Therefore, writing $t_1=2,$ we have
\begin{eqnarray*}
	\displaystyle
	I_2 &\le& C^{2j}\prod_{\ell=1}^{j+1} \|f_{\ell}\|^2_{W^{\alpha,2}}\int_1^\infty k^{(d-1)\left(\frac{1}{t_{j+1}}-\frac{1}{2}\right)+\sum_{\ell=1}^{j}\left(-2+(d-1)\left(\frac{1}{t_{\ell}}-\frac{1}{r_{\ell}}\right)\right)}dk
	\\
	&=&
	C^{2j}\prod_{\ell=1}^{j+1} \|f_{\ell}\|^2_{W^{\alpha,2}}\int_1^\infty
	\frac{dk}{k^{(d-1)\sum_{\ell=1}^{j}\left(\frac{1}{r_{\ell}}-\frac{1}{t_{\ell+1}}\right)+2j}}.
\end{eqnarray*}
Since $t_{\ell+1}$ and $r_{\ell}$ have to satisfy \eqref{monedero_4}, the best choice to get convergence of the previous integral is
\begin{equation}
\label{eleccion}
\frac{1}{r_{\ell}}-\frac{1}{t_{\ell+1}}=\frac{\alpha}{d}-\frac{1}{2}.
\end{equation}
With this choice we get
\begin{eqnarray}
\displaystyle
I_2 &\le& 
C^{2j}\prod_{\ell=1}^{j+1} \|f_{\ell}\|^2_{W^{\alpha,2}}\int_1^\infty
\frac{dk}{k^{(d-1)\left(\frac{\alpha}{d}-\frac{1}{2}\right)j+2j}}
\nonumber
\\
&\le&
C^{2j}\prod_{\ell=1}^{j+1} \|f_{\ell}\|^2_{W^{\alpha,2}},
\label{I2}
\end{eqnarray}
for any $j\in\mathbb{N}$ fixed, whenever 
\begin{equation}
\label{restriccion}
\alpha>d\left(\frac{1}{2}-\frac{1}{d-1}\right),
\end{equation}
and there exist $r_{\ell}$ and $t_{\ell+1}$ satisfying \eqref{monedero_3}, \eqref{monedero_4} and \eqref{eleccion}, for $\ell=1,2\ldots j$. Such $r_{\ell}$ and $t_{\ell+1}$ exist if
\begin{equation}
\label{restriccionI2}
0\le \alpha<d/2.
\end{equation}

Finally, the result follows from \eqref{I}, \eqref{I1} and \eqref{I2} if $\alpha\le 1$ and satisfies \eqref{restriccion}, \eqref{restriccionI2} and, \eqref{restriccionesI1a} or \eqref{restriccionesI1b}, that is, if $\alpha$ satisfies \eqref{st}.
\hfill
$\Box$

In order to prove Proposition \ref{resto}, we need the existence and uniqueness of solution of the direct scattering problem in $\in W_{-\delta}^{2,2}(\mathbb{R}^d)$ for any fixed wave number $k>0.$
This is equivalent to prove that there exists a unique $u_s,$ satisfying \eqref{Lippmann}.
For completeness we present here that result. The following lemma will be needed (see Theorem 6.5 in \cite{R2} or \cite{KK}).
\begin{lemma}
	\label{AlbertoT6.5}
	Let $k>0.$ There exists $\delta>1$ such that 
	\begin{equation}
	\label{Alberto}
	\|R_kf\|_{W_{-\delta}^{2,2}}\le c_k\|f\|_{L_{\delta}^2}.
	\end{equation}
\end{lemma} 
\begin{theorem}
	\label{tpds}
	For  $\alpha$ satisfying $\alpha>0$ and $\alpha>\frac{d}{2}-2$,
	let $q\in W^{\alpha,2}(\mathbb{R}^d)$ be real valued and compactly supported. 
	For any given $k>0$ and $\theta_0\in S^{d-1},$ there exists a unique solution $u_s=u_s(x,\theta_0,k)$ of equation \eqref{Lippmann} such that $u_s\in W_{-\delta}^{2,2}(\mathbb{R}^d)$ for certain $\delta>1$.
\end{theorem}
\begin{proof}
	We begin by proving the uniqueness. In order to do that, we have to prove that if $u$ is a solution in $W_{-\delta}^{2,2}(\mathbb{R}^d)$ of $u(x)=R_k(qu)(x)$ with $x\in\mathbb{R}^d$, then $u(x)=0$ for all $x\in\mathbb{R}^d.$ 
	
	We have that, in a weak sense
	\begin{equation}
	\label{homogenea}
	\left(\Delta  + k^2\right) u(x)=q(x)u(x), \qquad x\in\mathbb{R}^d,
	\end{equation}
	and $u$ satisfies the outgoing Sommerfeld condition.
	
	Consider $M>0$ large enough such that $supp\ q\subset B_M=B(0,M).$ 
	Multiplying (\ref{homogenea}) by $\overline{u}$ and integrating by parts we get
	\begin{equation*}
	\label{partes}
	-\int_{B_M}|\nabla u|^2
	+\int_{S_M}\partial_r u\ \overline{u}
	+k^2\int_{B_M}|u|^2
	=
	\int_{B_M}q|u|^2,
	\end{equation*}
	where $S_M=\partial B_M.$
	
	From here,
	since $q$ is real we have that
	\begin{equation}
	\label{Rellich}
	\Im \int_{S_M}\partial_r u\ \overline{u}=0.
	\end{equation}
	Writing the outgoing Sommerfeld condition given in (\ref{Sommerfeld}) in the following equivalent form
	\begin{equation*}
	\label{SommerEquivalente}
	\lim_{M\rightarrow\infty}\int_{S_M}\left|\partial_ru-iku\right|^2=0,
	\end{equation*}
	and using (\ref{Rellich}) we get
	$$
	\lim_{M\rightarrow\infty}\int_{S_M}|u|^2=0.
	$$
	From here, since $u$ is a radiating solution of equation $\left(\Delta  + k^2\right) u=0$ in the exterior domain $\mathbb{R}^d/B_M,$ we conclude using Rellich's lemma (see Lemma 2.11 in page 32 of \cite{ColtonKress}) that $u=0$ out of $B_M,$
	and by a unique continuation argument (see \cite{JerisonKenig}), $u(x)=0$ for all $x\in \mathbb{R}^d.$ 
		
	To prove existence of solution we will use Fredholm alternative.
	We introduce the operator $T_k$ defined by
	$T_ku=R_k(qu).$ 
	From \eqref{Alberto}, since $supp\ q\subset B_M,$
	we have that there exists $\delta>1$ such that 
	\begin{equation}
	\label{compacidad1}
	\|T_ku\|_{W_{-\delta}^{2,2}}\le C_{k,M}	\|qu\|_{L^2(B_M)}.
	\end{equation}
	On the other hand, from Sobolev embedding theorem we have that
	\begin{equation}
	\label{compacidad2}
	{W^{\beta,p}}(B_M)\underset{\text{compact}}{\hookrightarrow}L^2(B_M),
	\end{equation}
	whenever
	\begin{equation*}
	\beta\ge 0,\quad p>1,\quad \beta p<d,\quad 2<\frac{dp}{d-\beta p}.
	\end{equation*}
	Using \eqref{compacidad1}, \eqref{compacidad2} and \eqref{Z3}, we have that
	$T_k:W_{-\delta}^{2,2}\longrightarrow W_{-\delta}^{2,2}$ is a linear compact operator
	whenever $\alpha>0$ and $\alpha>\frac{d}{2}-2$. Thus, Fredholm alternative ensures the existence of solution of 
	$(I-T_k)u=R_k(q e^{ik\theta_0\cdot(\cdot)}),$ 
	since
	we have proved that any solution of 
	$(I-T_k)u=0$ in $\mathbb{R}^d$ satisfying the outgoing Sommerfeld radiation conditions must be the trivial solution,
	and $R_k(q e^{ik\theta_0\cdot(\cdot)})\in W_{-\delta}^{2,2}(\mathbb{R}^d).$
	\end{proof}
We will also need to know the behaviour of the scattered solution $u_s(x,\theta_0,k)$ when $k$ tends to zero (zero energy case). The following lemma, which can be found in \cite{Ru1}, collects several results given in  \cite{J1,J2,JK,M}.
\begin{lemma}{\rm{(\cite[Proposition 2.1]{Ru1})}}
	Let $q\in L^2(\mathbb{R}^d) \cap L^r(\mathbb{R}^d)$ for some $r>d/2,$ be real valued and compactly supported, and let $u_s(x, \theta_0, k)$ be the solution of \eqref{Lippmann}. For $\delta >1$ and , we have that
\begin{enumerate}
\item If $d \geq 3$,
\begin{equation}\label{murata_1}
\|u_s\|_{W^{2,2}_{-\delta} }=O(k^{-1}), \hspace{0.3cm} k \longrightarrow 0.
\end{equation}
\item If $d=2$,
\begin{equation}\label{murata_2}
\|u_s\|_{W^{2,2}_{-\delta} }=O((k \log k)^{-1}), \hspace{0.3cm} k \longrightarrow 0.
\end{equation}
\end{enumerate}
\end{lemma}

\emph{Proof of Proposition \ref{resto}.}
Arguing as in the proof of Proposition \ref{proposicion_apendice}, we split the norm to control into two pieces $I$ and $II,$ and we will just bound $I,$ since $II$ can be bounded in a similar way.
In this case, making the change of variables given in \eqref{cambio}, we write
\begin{eqnarray}
I&=&\int_{H_{\theta_0}}\left\langle\xi\right\rangle^{2\alpha} \left|\widehat{q_m^r}(\xi)\right|^2 d \xi
\nonumber
\\
&\le&
C\int_0^{\frac{1}{2}} k^{d-1}\int_{S^{d-1}} \left| \widehat{q_m^r}(k(\theta-\theta_0)\right|^2  d\sigma (\theta) dk
\nonumber
\\
&&
+C \int_{\frac{1}{2}}^\infty k^{d-1+2\alpha}\int_{S^{d-1}} |\theta-\theta_0|^{2\alpha} \left| \widehat{q_m^r}(k(\theta-\theta_0)\right|^2  d\sigma (\theta) dk
\nonumber
\\
\label{laIqrm}
&=&C(I_1+I_2),
\end{eqnarray}
whenever $0\le\alpha\le 1.$

To estimate $I_1,$ from the definition of $q_m^r$ given in \eqref{resto_m}, using Cauchy-Schwarz inequality, since $q$ is compactly supported, we get
$$ 
\left| \widehat{q_m^r}(k(\theta-\theta_0)\right|  
\le
\left\|(q R_k)^m(qu_s)\right\|_{L^1}    
\le
C \left\|(q R_k)^m(qu_s)\right\|_{L^2_\delta  }.$$
We can bound the $L^2_\delta$-norm that appears in the last inequality using \eqref{zolesioI1} and \eqref{zapatero_1} as the following diagram illustrates
\begin{equation*}
\begin{CD}
W_{-\delta}^{2,2}
@>
{\texttt{\rm{Zolesio}}}
>>
L^2_\delta
\ldots
@>
{(m-1)-{\texttt{\rm{times}}}}
>>
L^2_\delta
@>
R_k
>>
W_{-\delta}^{2,2}
@>
{\texttt{\rm{Zolesio}}}
>>
L^2_\delta
.
\end{CD}
\end{equation*}
Therefore, if  $\delta>1$ and $\alpha$ satisfies \eqref{restriccionesI1a} or \eqref{restriccionesI1b} we get
$$
\left\|(q R_k)^m(qu_s)\right\|_{L^2_\delta  }\le C^{m}  \left\| q \right\|^{m+1}_{W^{\alpha ,2}}  \left\|
  u_s  \right\|_{W^{2,2}_{-\delta}},
  $$
thus,
$$I_1 
\le
C^{2m}  \left\| q \right\|^{2(m+1)}_{W^{\alpha ,2}}  \int_0^{\frac{1}{2}}k^{d-1} \left\|
  u_s(\cdot , \theta_0,k)  \right\|^2_{W^{2,2}_{-\delta}} dk. 
  $$
Since $q\in W^{\alpha ,2}(\mathbb{R}^d)$, Sobolev embedding theorem guarantees that $q\in L^2(\mathbb{R}^d) \cap L^r(\mathbb{R}^d)$ with $r\in\left(\frac{d}{2},\frac{2d}{d-2\alpha}\right]$
whenever $\alpha\ge 0$ and $\frac{d}{2}-2\le\alpha\le\frac{d}{2}$, so we can use estimates (\ref{murata_1}) and (\ref{murata_2}) in the previous inequality to obtain 
\begin{equation}\label{I1qrm}
I_1\le  C^{2m}  \left\| q \right\|^{2(m+1)}_{W^{\alpha ,2}}.
\end{equation}

On the other hand, denoting $w=|\theta-\theta_0|^\alpha,$ we can write
\begin{equation}\label{integral_4_qrm}
I_2 = \int_{\frac{1}{2}}^\infty k^{d-1+2 \alpha}\left\| w\,  \widehat{q^r_m}(k(\theta-\theta_0))\right\|^2_{L^2(S^{d-1})}dk.
\end{equation}
Arguing as in the proof of Proposition \ref{proposicion_apendice}, to control $I_2$ we write $q^r_m$ in terms of the operators $R_{k,\theta}$ and  $S_{k,\theta_0}$ introduced in \eqref{Rconjugado} and \eqref{def:restriction} respectively, 
\begin{equation*}
\widehat{q^r_m}(k(\theta-\theta_0))
=S_{k,\theta_0}\left((qR_{k,\theta_0})^m\left( e^{-ik\theta_0 \cdot (\cdot)}q u_s    \right)\right) (\theta).
\end{equation*}
From here, using \eqref{lema2}, \eqref{zolesioI2} and \eqref{lema1} as the following diagram illustrates
\begin{equation*}
\begin{CD}
W^{\alpha, t_1}_\delta
\ldots
@>
(m-1){\texttt{\rm{\tiny{times}}}}
>>
W^{\alpha, t_m}_{\delta}
@>
R_{k,\theta_0}
>>
W^{\alpha,r_m}_{-\delta}
@>
{\texttt{\rm{\tiny{Zolesio}}}}
>>
W^{\alpha,t_{m+1}}_{\delta} 
@>
S_{k,\theta_0}
>>
L^2(S^{d-1}),
\end{CD}
\end{equation*}
where $0\le\frac{1}{t_1}-\frac{1}{2}\le\frac{1}{d+1}$ and,
$r_{\ell}$ and $t_{\ell+1}$ satisfy \eqref{monedero_3} and \eqref{monedero_4} for $\ell=1,2\ldots m,$ 
we get
\begin{equation}
\label{normS}
\left\| w\,  \widehat{q^r_m}(k(\theta-\theta_0))\right\|^2_{L^2(S^{d-1})}
\le
C^{2m}
k^a
\left\| q\right\|^{2m}_{W^{\alpha, 2}}
\left\| e^{-ik\theta_0 \cdot (\cdot)}q u_s\right\|^2_{W^{\alpha, t_1}_\delta},
\end{equation}
with
\begin{equation*}
\label{a}
a=(d-1)\left(\frac{1}{t_{m+1}}-\frac{3}{2}\right)-2\alpha+\sum_{\ell=1}^{m}\left(-2+(d-1)\left(\frac{1}{t_{\ell}}-\frac{1}{r_{\ell}}\right)\right).
\end{equation*}
We can control the norm on the right hand side of \eqref{normS} multiplying \eqref{Lippmann} by $q(x)e^{-ik\theta_0 \cdot x}$, using the operator $R_{k,\theta_0}$, the triangular inequality and estimate \eqref{zolesioI2} to write
\begin{equation}
\left\| e^{-ik\theta_0 \cdot (\cdot)}q u_s\right\|_{W^{\alpha, t_1}_\delta}
\le
C
\|q\|_{W^{\alpha, 2}}
\left(\left\| R_{k,\theta_0}q\right\|_{W^{\alpha, r_0}_{-\delta}}
+
\left\|R_{k,\theta_0}( e^{-ik \theta_0 \cdot (\cdot)}q u_s)\right\|_{W^{\alpha, r_0}_{-\delta}}\right)
\label{derecha},
\end{equation}
whenever
\begin{equation*}\label{monedero_6}
\alpha\ge 0,\quad t_{1} < \min(2, r_0)\quad \text{   and   } \quad 0 \leq\frac{1}{2}+\frac{1}{r_0}-\frac{1}{t_{1}} \leq \frac{\alpha}{d}.
\end{equation*}  
Using \eqref{lema1} we get
$$
\left\|R_{k,\theta_0}( e^{-ik \theta_0 \cdot (\cdot)}q u_s)\right\|_{W^{\alpha, r_0}_{-\delta}}
\le
k^bC
\left\| e^{-ik \theta_0 \cdot (\cdot)}q u_s\right\|_{W^{\alpha, t_1}_\delta},
$$
with 
$$
b=-1+\frac{(d-1)}{2}\left(\frac{1}{t_1}-\frac{1}{r_0}\right),
$$
whenever
\begin{equation}\label{monedero_5}
\alpha\ge 0,\quad 0 \leq \frac{1}{t_{1}}- \frac{1}{2} \leq \frac{1}{d+1}\quad\text{  and  } \quad  0 \leq \frac{1}{2}- \frac{1}{r_0} \leq \frac{1}{d+1}, 
\end{equation}
From here, using \eqref{peque}, since $b<0,$ then
for $k\ge 1/2,$ we have that
\begin{eqnarray}
\label{absorcion}
\|q\|_{W^{\alpha, 2}}
\left\|R_{k,\theta_0}( e^{-ik \theta_0 \cdot (\cdot)}q u_s)\right\|_{W^{\alpha, r_0}_{-\delta}}
&\le&
\frac{A}{2^b}\,C
\left\| e^{-ik \theta_0 \cdot (\cdot)}q u_s\right\|_{W^{\alpha, t_1}_\delta}
\nonumber
\\
&<&
\frac{1}{2}\left\| e^{-ik \theta_0 \cdot (\cdot)}q u_s\right\|_{W^{\alpha, t_1}_\delta}
\end{eqnarray}
whenever $AC\le 2^{b-1}.$

Using \eqref{absorcion} in \eqref{derecha} we obtain
\begin{equation*}\label{derecha2}
\left\| e^{-ik\theta_0 \cdot (\cdot)}q u_s\right\|_{W^{\alpha, t_1}_\delta}
\le
2C
\|q\|_{W^{\alpha, 2}}
\left\| R_{k,\theta_0}q\right\|_{W^{\alpha, r_0}_{-\delta}}.
\end{equation*}
From here, since $0 \leq \frac{1}{2}- \frac{1}{r_0} \leq \frac{1}{d+1}$ (see\eqref{monedero_5}) we can use \eqref{lema1}, to get
\begin{equation*}\label{derecha3}
\left\| e^{-ik\theta_0 \cdot (\cdot)}q u_s\right\|_{W^{\alpha, t_1}_\delta}
\le C
k^{-1+\frac{(d-1)}{2}\left(\frac{1}{2}-\frac{1}{r_0}\right)}
\left\| q\right\|_{W^{\alpha, 2}}\left\| q\right\|_{W_{\delta}^{\alpha, 2}}.
\end{equation*}
And since $q$ has compact support, we can use \eqref{zolesioI2} to obtain
\begin{equation*}\label{derecha33}
\left\| e^{-ik\theta_0 \cdot (\cdot)}q u_s\right\|_{W^{\alpha, t_1}_\delta}
\le C
k^{-1+\frac{(d-1)}{2}\left(\frac{1}{2}-\frac{1}{r_0}\right)}
\left\| q\right\|^2_{W^{\alpha, 2}}.
\end{equation*}
Using this inequality and \eqref{normS} in \eqref{integral_4_qrm} we get
$$
I_2 
\le
C^{2m}\|q\|^{2(m+2)}_{W^{\alpha,2}}
\int_{\frac{1}{2}}^\infty k^{-2(m+1)-(d-1)\sum_{\ell=0}^{m}\left(\frac{1}{r_{\ell}}-\frac{1}{t_{\ell+1}}\right)}dk,
$$
for any $r_{\ell}$ and $t_{\ell+1}$ satisfy \eqref{monedero_3} and \eqref{monedero_4} for $\ell=0,1\ldots m.$ 

If we choose $r_{\ell}$ and $t_{\ell+1}$ satisfying \eqref{eleccion}, the last integral is convergent if and only if 
$$
2(m+1)+(d-1)(m+1)\left(\frac{\alpha}{d}-\frac{1}{2}\right)>1,
$$
or equivalently
\begin{equation*}
\label{sconm}
\alpha>\frac{d}{2}-\frac{(2m+1)d}{(m+1)(d-1)}.
\end{equation*}
Therefore, for any $m\in\mathbb{N}$, we have that 
\begin{equation}
\label{I2qrm}
I_2 
\le
C^{2m}\|q\|^{2(m+1)}_{W^{\alpha,2}}
\end{equation}
whenever 
\begin{equation*}
\label{sI2qrm}
0\le \alpha<d/2\qquad\text{and}\qquad \alpha>\frac{d}{2}-\frac{2d}{d-1}.
\end{equation*}
Finally, the result follows from \eqref{laIqrm},\eqref{I1qrm} and \eqref{I2qrm} for $\alpha$ satisfying \eqref{sqrm}.
\hfill
$\Box$
\section{Numerical experiments}

In this section we show two numerical experiments in dimension 2 that illustrate the efficiency of the proposed algorithm. We follow the discretization in \cite{BCR}, based on a trigonometric collocation method.  Note that the numerical version of the sequence (\ref{def:sequence}) only requires the numerical approximation of some integral equations involving powers of the resolvent $R_k$ (see the right hand side of (\ref{Q_operator})-(\ref{resto_m})) and the inverse Fourier transform. Both ingredients are described in \cite{BCR} and therefore the adaptation to the new algorithm defined here is straightforward. We refer to  \cite{BCR} for implementation details. 

We have considered two different examples corresponding to a piecewise constant potential (Example 1) and a smooth one (Example 2). Note that the piecewise constant potential in Example 1 does not satisfy the regularity condition in Theorem \ref{th:cauchy_sequence}, but nevertheless the algorithm provides good results. 
 
To compute the scattering data, i.e. the far field pattern, we have used a mesh twice finer than the mesh used to solve the inverse problem. In this way we try to simulate real data to recover the potential.

In the experiments below we have considered a computational domain $[-2.1,2.1]\times [-2.1,2.1]$ with $N$ grid points in each variable uniformly distributed. The potential is supported in the region $ [-1,1]\times [-1,1]$.

\subsection{Example 1}
In the first example the potential is a piecewise constant function, given by
\begin{equation}
q(x_1,x_2)=\left\{ 
\begin{array}{ll}
1.2, & \mbox{ if $|x_1|+|x_2|<0.3,$}\\
1, & \mbox{ if $0.7< |(x_1,x_2)|<1,$}\\
0, & otherwise.
\end{array}
\right.
\end{equation}
The error of the approximation is computed using the discrete $L^2$-norm of the difference between the numerical approximation of $q$ and the projection of the continuous functional in the mesh. 

In Figure \ref{fig:sol2} we show the behaviour of the error for the discrete approximations of $q_{m,\ell}$ in log-scale when considering a mesh with $32\times 32$ points.
Each line corresponds to $q_{m,\ell}$ for fixed $m$ and represents the error ($Y$-axis) with respect to $\ell$ ($X$-axis). Thus, the first line (in blue) illustrates the error of $q_{1,\ell}$. We see that it decreases as $\ell$ grows and it attains the lower error for $\ell=3$. 

The second line (in red) corresponds to $q_{2,\ell}$. It becomes stable for $\ell=4$. The other lines correspond to $q_{3,\ell}$ and $q_{4,\ell}$. We see that they exhibit the same behaviour as $q_{2,\ell}$. Thus, we deduce that for $m=2$ and $\ell=3$ we almost attain the minimal error.  
\begin{figure}
	\begin{center}
	\psfrag{-1}{$\hspace{-0.1cm}-1$}\psfrag{-0.9}{$-0.9$}\psfrag{-0.8}{$-0.8$}\psfrag{-0.7}{$-0.9$}\psfrag{-0.6}{$-0.6$}
		\psfrag{1}{$1$}\psfrag{2}{$2$}\psfrag{3}{$3$}\psfrag{4}{$4$}\psfrag{5}{$5$}\psfrag{6}{$6$}
		\psfrag{m=1}{$\hspace{-0.05cm}m\hspace{-0.1cm}=\hspace{-0.1cm}1$}
		\psfrag{m=2}{$\hspace{-0.05cm}m\hspace{-0.1cm}=\hspace{-0.1cm}2$}
		\psfrag{m=3}{$\hspace{-0.05cm}m\hspace{-0.1cm}=\hspace{-0.1cm}3$}
		\psfrag{m=4}{$\hspace{-0.05cm}m\hspace{-0.1cm}=\hspace{-0.1cm}4$}
		\psfrag{b}{\hspace{-1cm}$\log\|q_{m.\ell}-q\|_{L^2}$}
		\psfrag{a}{\hspace{-1cm}Parameter $\ell$}
		\includegraphics[height=7cm]{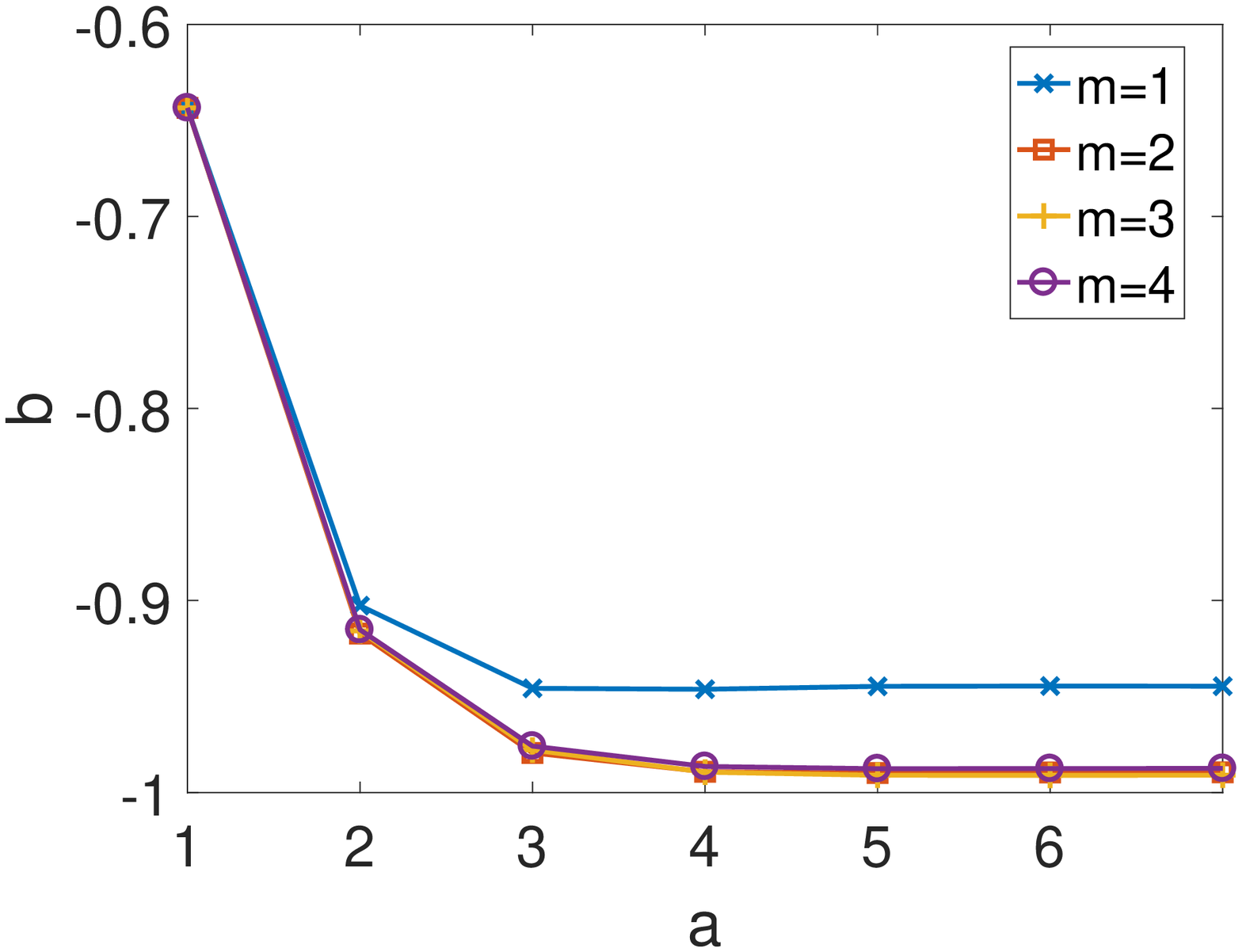}
		\caption{Error in log-scale for Example 1 with a mesh of $32\times32$ points.}
			\label{fig:sol2}
	\end{center}
\end{figure}
In Figure \ref{fig:sol3} we illustrate the same as in Figure \ref{fig:sol2} but this time with a mesh grid containing $128\times 128$ points. The behaviour is almost the same but the error becomes smaller. 
\begin{figure}
	\begin{center}
	\psfrag{-1.8}{$-1.8$}\psfrag{-1.6}{$-1.6$}\psfrag{-1.4}{$-1.4$}
		\psfrag{-1.2}{$-1.2$}\psfrag{-1}{$\hspace{-0.1cm}-1$}
		\psfrag{2}{$2$}\psfrag{3}{$3$}\psfrag{4}{$4$}\psfrag{5}{$5$}\psfrag{6}{$6$}\psfrag{7}{$7$}
		\psfrag{m=1}{$\hspace{-0.05cm}m\hspace{-0.1cm}=\hspace{-0.1cm}1$}
		\psfrag{m=2}{$\hspace{-0.05cm}m\hspace{-0.1cm}=\hspace{-0.1cm}2$}
		\psfrag{m=3}{$\hspace{-0.05cm}m\hspace{-0.1cm}=\hspace{-0.1cm}3$}
		\psfrag{m=4}{$\hspace{-0.05cm}m\hspace{-0.1cm}=\hspace{-0.1cm}4$}
		\psfrag{a}{\hspace{-1cm}$\log\|q_{m.\ell}-q\|_{L^2}$}
		\psfrag{b}{\hspace{-1cm}Parameter $\ell$}
		\includegraphics[height=7cm]{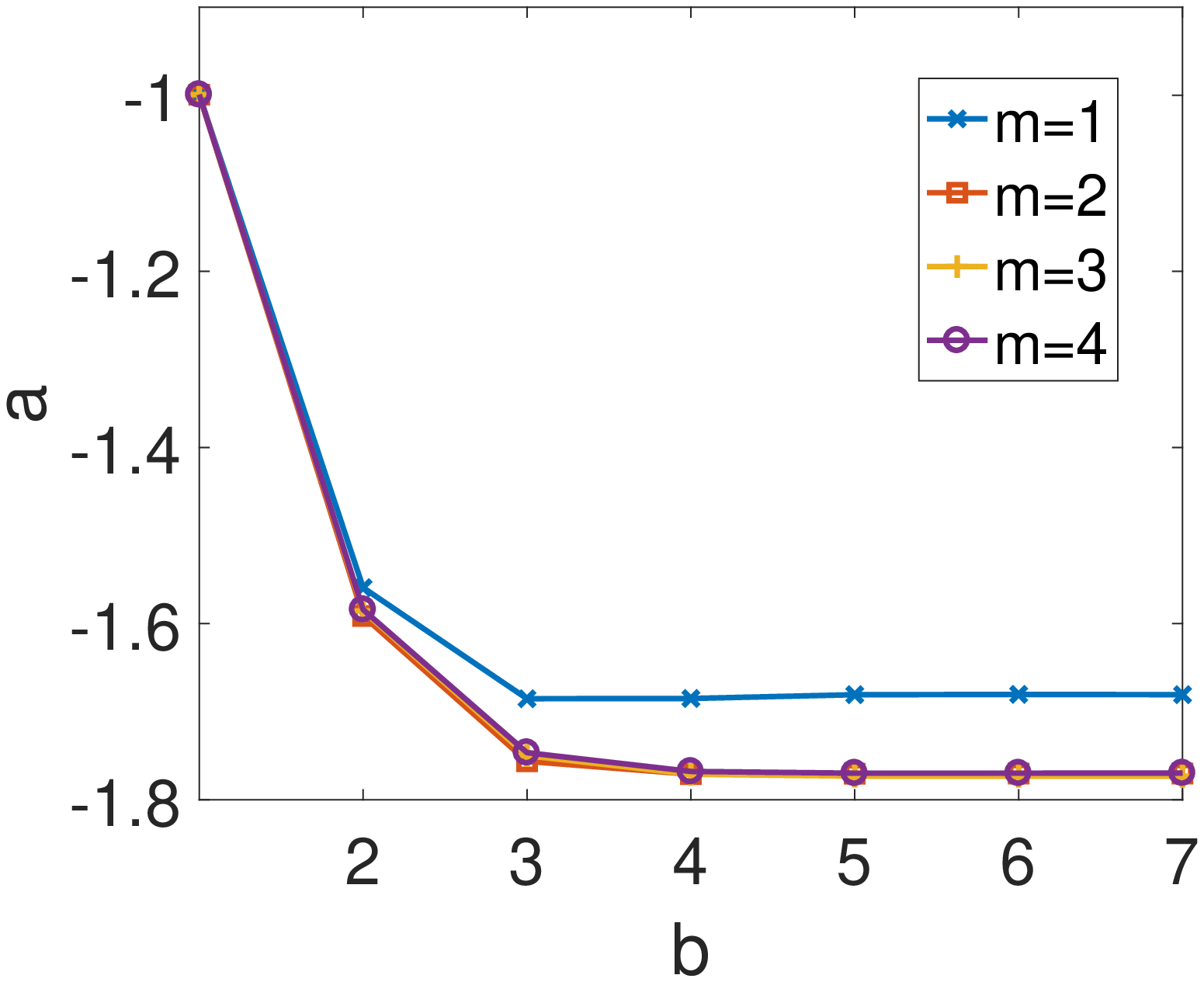}
		\caption{Error in log-scale for Example 1 with a mesh of $128\times128$ points.}
			\label{fig:sol3}
	\end{center}
\end{figure}

In Figure \ref{fig:sol3.1} we show a section at $x_2=0,$ of the graph of the potential given in Example 1 and its approximation $q_{4,7}$ considering a mesh with $N=128.$
\begin{figure}
	\begin{center}
	\psfrag{-2}{$-2$}\psfrag{-1}{$-1$}\psfrag{0}{$0$}
		\psfrag{2}{$2$}\psfrag{1}{$1$}
		\psfrag{0.5}{$0.5$}\psfrag{-0.5}{$-0.5$}
		\psfrag{1.5}{$1.5$}
		\includegraphics[height=7cm]{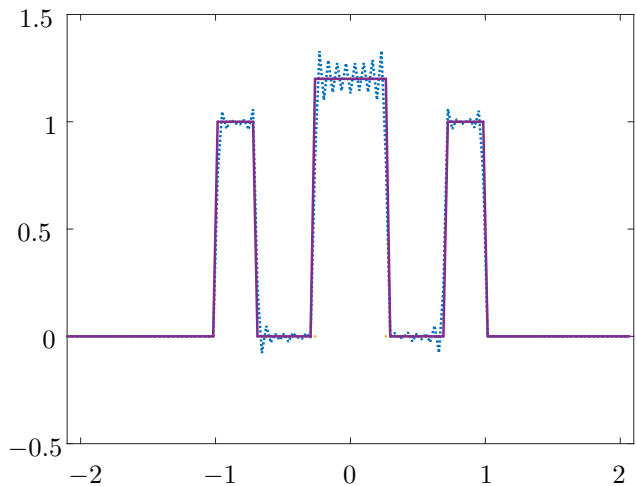}
		\caption{Section at $x_2=0$ of the graph of $q$ in Example 1 and its approximation $q_{4,7}$.}
			\label{fig:sol3.1}
	\end{center}
\end{figure}

Finally, in Figure \ref{fig:sol4} we illustrate the behaviour of the error as we consider finer meshes. To this end, we take $q_{4,7}$ as the best approximation for each mesh grid, since we have seen that larger values of $m$ and $\ell$ do not improve the error significatively. Then, we compare the error as the mesh becomes finer. This illustrates the convergence  of the approximations as $N$ goes to infinity. 
\begin{figure}
	\begin{center}
	\psfrag{-1.8}{$-1.8$}\psfrag{-1.6}{$-1.6$}\psfrag{-1.4}{$-1.4$}
		\psfrag{-1.2}{$-1.2$}\psfrag{-1}{\hspace{-0.1cm}$-1$}\psfrag{-0.8}{$-0.8$}
		\psfrag{32}{$32$}\psfrag{64}{$64$}\psfrag{128}{$128$}
		\psfrag{N}{\hspace{-1cm}Parameter $N$}
		\psfrag{b}{\hspace{-1cm}$\log\|q_{4,7}-q\|_{L^2}$}
		\psfrag{a}{\hspace{-1cm}Parameter $N$}
		\includegraphics[height=7cm]{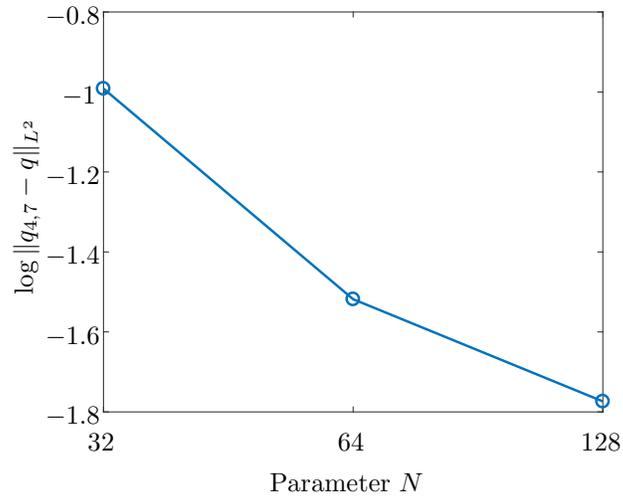}
		\caption{Error of the approximation $q_{4,7}$ in Example 1 for different meshes.}
			\label{fig:sol4}
	\end{center}
\end{figure}

\subsection{Example 2}
Here we consider a smooth potential given by 
\begin{eqnarray*}
	q(x_1,x_2)&=&\max(0,e^{-5|x_1-0.5|^2})+1.5e^{-4|(x_1,x_2)-(-1,0.8)/2|^2}\\&&+2e^{-7|(x_1,x_2)-0.4(-1,-1)|^2-0.4}.
\end{eqnarray*}

The numerical results for this example are completely similar to those of the previous one, but the errors are significatively smaller. This is due to the regularity of the potential $q$ that we recover. We illustrate this in Figure \ref{fig:sol5}, which is the analogous to Figure \ref{fig:sol4}. Note that in this case we obtain errors of $10^{-4.4}$ instead of $10^{-1.8}$ of the previous example.
\begin{figure}
	\begin{center}
	\psfrag{-4.4}{$-4.4$}\psfrag{-4.3}{$-4.3$}\psfrag{-4.2}{$-4.2$}
		\psfrag{-4.1}{$-4.1$}\psfrag{-4}{\hspace{-0.1cm}$-4$}\psfrag{-3.9}{$-3.9$}\psfrag{-3.8}{$-3.8$}
		\psfrag{32}{$32$}\psfrag{64}{$64$}\psfrag{128}{$128$}
			\psfrag{b}{\hspace{-1cm}$\log\|q_{4,7}-q\|_{L^2}$}
		\psfrag{a}{\hspace{-1cm}Parameter $N$}
		\includegraphics[height=7cm]{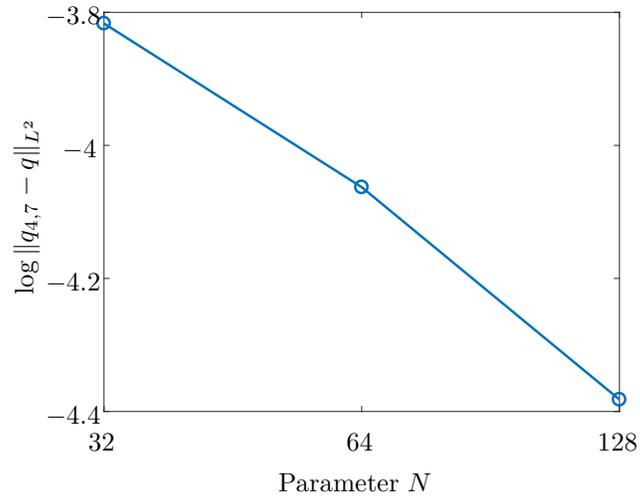}
		\caption{Error of the approximation $q_{4,7}$ in Example 2 for different meshes.}
			\label{fig:sol5}
	\end{center}
\end{figure}

On the other hand, the norm of the potential is not particularly small in this example, which shows that the smallness hypothesis in Theorem \ref{th:cauchy_sequence} can be probably relaxed in practice.   

Finally, it is worth mentioning that the numerical approximations obtained with the algorithm introduced here are similar to those obtained by the iterative one in \cite{BCR}, at least in the experiments that we have considered. The main point is that the computational cost is reduced drastically for large $N$ since, as we said in the introduction, we do not have to solve a Lipmann-Schwinger equation for each point in the mesh, and for each iteration. 
\section*{Acknowledgements}
The first and the second author were supported by Spanish Grant MTM2014-57769-C3-2-P,
	the third by Spanish Grant MTM2014-53850-P2, and
	the fourth by Spanish Grant
	MTM2014-57769-C3-1-P.

The authors would like to gratefully thank Alberto Ruiz for his invaluable advice. 
We also thank Samuli Siltanen and Juan Manuel Reyes.
\section*{References}


\end{document}